\documentclass[12pt]{amsart}
\usepackage{times}
\usepackage{amsfonts, amsmath, amssymb, amsthm, mathtools, dsfont}
\usepackage[margin=3cm]{geometry}
\usepackage[colorlinks=true,citecolor=blue]{hyperref}
\usepackage{setspace, enumerate}
\usepackage{comment, todonotes}
\usepackage[title]{appendix}
\usepackage{hyperref}
\usepackage[all, cmtip]{xy}

\theoremstyle{plain}
\newtheorem{theorem}{Theorem}[section]
\newtheorem{proposition}[theorem]{Proposition}
\newtheorem{lemma}[theorem]{Lemma}

\newtheorem{claim}[theorem]{Claim}
\newtheorem{question}[theorem]{Question}

\theoremstyle{definition}
\newtheorem{example}[theorem]{Example}

\title{Domination numbers and homotopy in certain ternary graphs}

\author{Taehyun Eom}
\address{
R\&BD FOUNDATIONS\\
Chonnam National University\\
Gwangju\\
Republic of Korea}
\email{teaheom@chonnam.ac.kr}

\author{Jinha Kim}
\address{Department of Mathematics\\
Chonnam National University\\
Gwangju\\
Republic of Korea 
\and Discrete Mathematics Group\\
Institute for Basic Science (IBS)\\
Daejeon\\
Republic of Korea}
\email{jinhakim@jnu.ac.kr}

\author{Minki Kim}
\address{Department of Mathematical Sciences\\
Gwangju Institute of Science and Technology\\
Gwangju\\
Republic of Korea }
\email{minkikim@gist.ac.kr}



\date{\today}

\begin{document}
\maketitle
\begin{abstract}
A ternary graph is a graph with no induced cycles of length $0$ modulo $3$.
It was recently shown that, if the independence complex of a ternary graph is not contractible, then it is homotopy equivalent to a sphere.
When a ternary graph also does not contain induced cycles of length $1$ modulo $3$, we prove that the dimension of the sphere is equal to the dimension of a minimum maximal simplex of the independence complex, or equivalently, to the value obtained by subtracting $1$ from the independent domination number of the graph.
The same statement holds if we replace the independent domination number with the domination number.
We also give a hypergraph analogue of the statement above.
\end{abstract}

\section{Introduction}
\noindent We reveal a relation between domination numbers and the homotopy type of the independence complex of certain graphs, extending a known result for forests.
Throughout the paper, all graphs are simple and undirected.

Given a graph $G$ on $V$, a vertex set containing no adjacent pair is said to be {\em independent}. 
Being an independent set of $G$ is hereditary under taking a subset, so we can define the (abstract) simplicial complex on $V$ whose simplices are the independent sets of $G$.
It is called the {\em independence complex} of $G$, and we denote it by $I(G)$.

About a decade ago, Kalai and Meshulam~\cite{KM} posed a conjecture that the total Betti number of the independence complex of a graph with no induced cycles of length $0$ modulo $3$ is at most $1$.
Such a graph is called a {\em ternary graph}.
The above conjecture was recently shown by Zhang and Wu~\cite{WZ23}, and Kim~\cite{Kim22} proved a strengthened statement that says the independence complex of a ternary graph is either contractible or homotopy equivalent to a sphere.

For a ternary graph $G$, let 
\begin{align}\label{d(G)}
d(G)=\begin{cases}
    d & \text{if }I(G)\simeq S^{d}\text{ for some }d\geq0,\\
    * & \text{if }I(G)\text{ is contractible}.
\end{cases}    
\end{align}
It is natural to ask the following question.
\begin{question}
    Let $G$ is a ternary graph, what is $d(G)$?
\end{question}

We give an answer to this question for a subclass of ternary graphs in terms of domination numbers of the graphs.
Let $G$ be a graph on $V$.
For $A\subset V$ and $v \in V$, we say $A$ {\em dominates} $v$ if either $v\in A$ or $v$ has a neighbor of some vertex of $A$.
If $A$ dominates $V$, we also say $A$ {\em dominates} $G$ or $A$ is a {\em dominating set} of $G$.
The {\em domination number}, denoted by $\gamma(G)$, is the size of a minimum dominating set of $G$.
When a dominating set of $G$ is independent, we call it an {\em independent dominating set} of $G$.
Similarly as above, the {\em independent domination number} is the size of a minimum independent dominating set of $G$, and we denote it by $i(G)$.
We call $W \subseteq V(G)$ is an \textit{$i(G)$-dominating set} when $W$ is a minimum independent dominating set of $G$.
Obviously, we have $i(G) \geq \gamma(G)$ for all graphs.

When $G$ is a forest, it was already known by Ehrenborg and Hetyei~\cite{EH06} that the independence complex of $G$ is either contractible or homotopy equivalent to a sphere.
Then Marietti and Testa~\cite{MT08} determined $d(G)$ by showing that either $d(G)=*$ or $i(G)=d(G)+1=\gamma(G)$.
This result also implies that $I(G)$ is contractible whenever $G$ is a forest with $i(G)\neq \gamma(G)$.
We establish a result that takes a step beyond the work of Marietti and Testa, by providing a broader graph class with the same property.

When a graph $G$ is either a path or a cycle, we denote the number of its edges, or the length of $G$, by $\ell(G)$.
We say that a graph $G$ is {\em $(0,1)$-ternary} if it has no induced cycles of length $0$ or $1$ modulo $3$.
Actually, such graphs contain no cycles of length $0$ or $1$ modulo $3$.
\begin{proposition}
    $G$ is $(0,1)$-ternary if and only if it has no cycles of length $0$ or $1$ modulo $3$.
\end{proposition}
\begin{proof}
    The if part is obvious.
    Suppose $G$ has no induced cycles of length $0$ or $1$ modulo $3$.
    Assume there is a cycle of length $0$ or $1$ modulo $3$ in $G$.
    Take one, say $C$, of minimum length among such.
    We claim that $C$ should be an induced cycle.
    Otherwise, $C$ has a chord $e$, and the endpoints of $e$ divides $C$ into two internally vertex-disjoint paths $P$ and $Q$.
    Since each $P+e$ and $Q+e$ are cycles, they should have length $2$ modulo $3$ by the minimality assumption on $C$.
    Then $\ell(C) = \ell(P+e)+\ell(Q+e)-2 \equiv 2 ~\text{(mod }3)$, which is a contradiction.
    Therefore, $G$ contains no cycles of length $0$ or $1$ modulo $3$.
\end{proof}

Our main result is the following.
\begin{theorem}\label{thm:main}
For every $(0,1)$-ternary graph $G$, if $d(G)\neq *$ then $$d(G)+1=i(G)=\gamma(G)=\gamma(L(G)).$$
\end{theorem}
The above theorem includes the result by Marietti and Testa, since every forest is $(0,1)$-ternary.
The proof of Theorem~\ref{thm:main} is based on the structure of $(0,1)$-ternary graphs, as restriction on the length of cycles makes the structure of such graphs quite simple.
Based on the structural result, we prove that $d(G)\geq i(G)-1$ and that $\gamma(L(G)) \leq \gamma(G) \leq i(G)$, where $L(G)$ is the line graph of $G$. It follows from \cite[Theorem~1.2]{HW14} that $\tilde{H}_i(I(G))=0$ for all $i \geq \gamma(L(G))$, 
so we have $d(G) \leq \gamma(L(G))-1$.
Then the inequality $\gamma(L(G)) \leq \gamma(G) \leq i(G)$ implies $d(G)+1=i(G)=\gamma(G)=\gamma(L(G))$.
From the argument, we see that the independence complex of a $(0,1)$-ternary graph $G$ is contractible whenever the values of $i(G)$, $\gamma(G)$, and $\gamma(L(G))$ do not match.

We also give a hypergraph analogue of Theorem~\ref{thm:main}.
A \textit{hypergraph} $H$ on a vertex set $V(H)$ is a collection of nonempty subsets of $V(H)$.
Each element of $H$ is called an \textit{edge} of $H$, and
an \textit{independent set} of $H$ is a a vertex subset that does not contain any edge as a subset. The \textit{independence complex} $I(H)$ of a hypergraph $H$ is a simplicial complex whose faces are the independent sets of $H$.
A \textit{Berge cycle} of length $k$ in a hypergraph $H$ is an alternating sequence $C=v_1 e_1 v_2 e_2 \ldots v_k e_k$ of distinct vertices and edges of $H$, where $v_i, v_{i+1} \in e_i$ for all $i$ taken modulo $k$.

For a simplicial complex $K,L$, we say $K$ and $L$ are \textit{homology equivalent} if $K$ and $L$ has the same homology group in all dimension, i.e. $\tilde{H}_i(K;\mathbb{Z})=\tilde{H}_i(L;\mathbb{Z})$ for all $i$.
It was shown in \cite{Kim24} that if a hypergraph $H$ has no Berge cycle of length $0$ modulo $3$, then $I(H)$ is homology equivalent to a contractible space or a sphere.
Hence, for a hypergraph $H$ with no Berge cycle of length $0$ modulo $3$, it is possible to define the following parameter.
\[d(H)=
\begin{cases}
    d & \text{ if } I(H) \text{ is homology equivalent to } S^{d}\text{ for some }d\geq0,\\ 
    * & \text{ if } I(H) \text{ is homology equivalent to a contractible space}.
\end{cases}\]
Note that the notation is consistent with \eqref{d(G)} when $H$ is a ternary graph because in that case, we already know that $I(H)$ is either contractible or homotopy equivalent to a sphere.

Given a hypergraph $H$, a vertex $v$ is dominated by a vertex set $W$ if $v\in W$ or $\{v\}\cup W'\in H$ for some $W'\subset W$.
If every vertex of $H$ is dominated by $W$, then we say $W$ is a {\em dominating set} of $H$.
The \textit{domination number} $\gamma(H)$ of $H$ is the minimum size of a dominating set of $H$.
We say a hypergraph $H$ is \textit{$(0,1)$-ternary} if $H$ has no Berge cycle of length $0$ or $1$ modulo $3$.
Then the following holds.
\begin{theorem}\label{thm:hypergraph}
    For every $(0,1)$-ternary hypergraph $H$, either $d(H)=*$ or $d(H)=\gamma(H)-1$.
\end{theorem}

\subsection{Organization}
In Section~\ref{sec:ear-decomp}, we prove some structural properties of a $2$-connected graphs where all cycles have the same length modulo $m$ for some odd $m$. Then, in Section~\ref{sec:i(G)}, we give results on domination numbers of $(0,1)$-ternary graphs with a focus on the usage in the proof of Theorem~\ref{thm:main}. 
The proof of Theorem~\ref{thm:main} and Theorem~\ref{thm:hypergraph} will be presented in Section~\ref{sec:main}, and in Section~\ref{sec:remark}, we discuss what happens in arbitrary ternary graphs.

\section{Ear decomposition of $2$-connected graphs with restricted cycle lengths}\label{sec:ear-decomp}
A graph is {\em $2$-connected} if it has at least $3$ vertices and it does not contain a vertex whose removal disconnects the graph.
Let $k$ and $m$ be positive integers where $k<m$ and $m$ is odd.
Note that $2x \equiv k~\text{(mod }m)$ has a unique solution, say $k'$, up to modulo $m$.
In this section, we investigate some structural properties of $2$-connected graphs whose cycles have length $k$ modulo $m$.
We start with an observation on such graphs without the $2$-connected assumption.
\begin{lemma}\label{lem:cycles}
    Let $G$ be a graph such that every cycle of $G$ has length $k$ modulo $m$.
    Let $u$ and $v$ be two distinct vertices of $G$ and suppose that $G$ contains three internally vertex-disjoint paths $P_1$, $P_2$, $P_3$ of connecting $u$ and $v$.
    Then $G$ cannot have a path connecting an internal vertex of $P_i$ and an internal vertex of $P_j$ for any distinct $i,j \in \{1,2,3\}$.
\end{lemma}
\begin{proof}
    Suppose that there exists such a path $P$.
    Without loss of generality, assume that $P$ connects an internal vertex of $P_1$ and that of $P_2$.
    Consider the cycles $C_1$, $C_2$, $\ldots$, $C_7$ depicted in Figure~\ref{fig:lem_cycles}.
    \begin{figure}[htbp]
        \centering
        \includegraphics[scale=0.7]{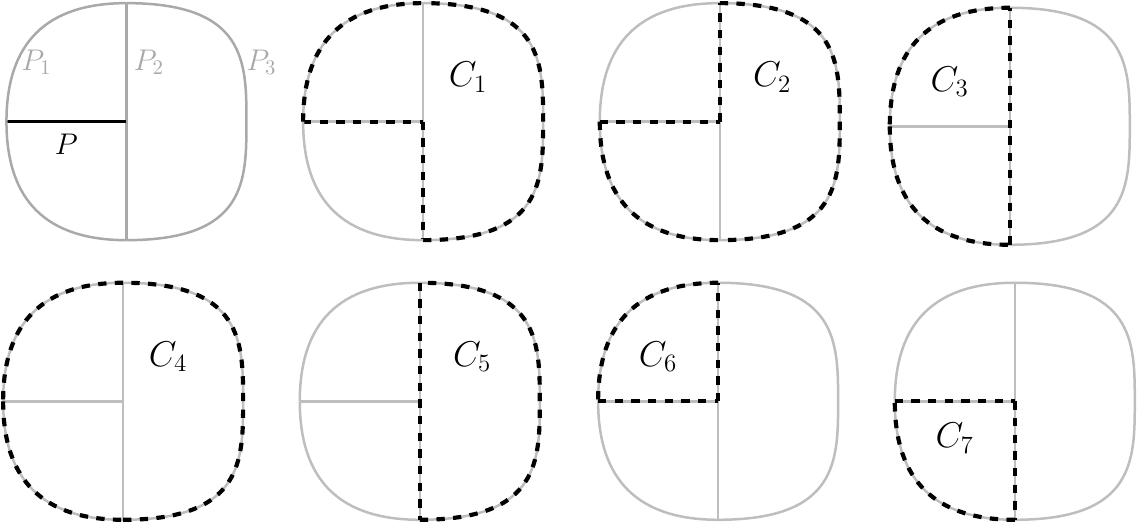}
        \caption{Cycles $C_1$, $C_2$, $\ldots$, $C_7$.}
        \label{fig:lem_cycles}
    \end{figure}
    
    By the assumption that every cycle in $G$ should have length $k$ modulo $m$, we have
    $$2\ell(P)=\ell(C_1)+\ell(C_2)-\ell(C_4)-\ell(C_5)\equiv0~\text{(mod }m).$$
    On the other hand, we also have
    $$2\ell(P)=\ell(C_6)+\ell(C_7)-\ell(C_3)\equiv k~\text{(mod }m).$$
    Since $k\neq0$, this is a contradiction.
    Therefore, there is no such path.
\end{proof}

Next, we show that if such a graph is $2$-connected, then any ear decomposition of the graph has a rooted tree-like structure.
\begin{theorem}\label{ear-decomp}
    Let $G$ be a $2$-connected graph where all cycles have length $k$ modulo $m$.
    Let $H$ be the graph obtained from $G$ by attaching a new ear $E$, preserving the restriction on the length of a cycle.
    Let $u$ and $v$ be the two endpoints of $E$.    
    Then 
    \begin{enumerate}
        \item any path connecting $u$ and $v$ in $H$ has length $k'$ modulo $m$, and
        \item $u$ and $v$ should be lying on the same ear of $G$.
    \end{enumerate}
\end{theorem}
\begin{proof}
    Since $G$ is $2$-connected, there exist internally vertex-disjoint paths \(P'\) and \(P''\) connecting \(u\) and \(v\) in \(G\). Then \(P'\cup E\), \(P''\cup E\), and \(P'\cup P''\) are cycles of \(H\), so they all have length \(k\) modulo \(m\). 
    This particularly implies $\ell(E)\equiv \ell(P')\equiv \ell(P'') \equiv k'~\text{(mod }m)$.
    As $P'$ and $P''$ are arbitrarily chosen, it is shown that any path connecting \(u\) and \(v\) in \(H\) has length \(k'\) modulo \(m\).
    This proves (1).

    Now we prove (2).
    Let $P_0$ be the initial cycle in the ear decomposition of $G$, and let $P_1,\ldots,P_t$ be the ears in the order in which they are attached.
    Note that each $P_i, i>0$ are paths. Let \(G_0=P_0\), \(G_1=G_0\cup P_1\), $\ldots$, \(G_t=G_{t-1}\cup P_t=G\).
    For each vertex \(w\), let \(i(w)\) be the first index \(i\) that satisfies \(w\in V(G_{i})\). Then it must be \(w\in P_i\) and \(\deg_{G_i}(w)=2\). Assume without loss of generality that \(i(u)\leq i(v)\). We will show that either \(i(u)=i(v)\) or \(u\) is an endpoint of \(P_{i(v)}\).
    Suppose \(i(u)<i(v)\) and \(u\) is not an endpoint of \(P_{i(v)}\). Let $x$ and $y$ be endpoints of \(P_{i(v)}\). 
    
    Since $G_{i(v)-1}$ is $2$-connected, it has a cycle $C$ containing $x$ and $y$.
    Let $C^+$ and $C^-$ be the two internally disjoint paths connecting $x$ and $y$ whose union is $C$.
    If $u$ is a vertex of the cycle $C$, then $C^+$, $C^-$, and $P_{i(v)}$ are three internally disjoint paths connecting $x$ and $y$, and the argument for (1) shows that each of the paths has length $k'$ modulo $m$.
    However, this is a contradiction since the path $E$ cannot be added by Lemma~\ref{lem:cycles}.
    Otherwise, let $C'$ be the cycle obtained as the union of $C^+$ and $P_{i(v)}$.
    By the $2$-connectivity of $G_{i(v)}$, there exist two internally disjoint paths $A$ and $B$ from $u$ to $v$.
    Let $x'$ and $y'$ be the vertices where $A$ and $B$ meets $C'$, respectively.
    Here, since \(\deg_{G_{i(v)}}(v) = 2\) and its neighbors are contained in \(P_{i(v)}\), vertices \(x',y',v\) are all distinct.
    Now $C' \cup A \cup B$ can be decomposed into three internally disjoint paths $Q_1$, $Q_2$, $Q_3$ connecting $x'$ and $y'$ such that $u$, $v$ are internal vertices of $Q_1$, $Q_2$, respectively.
    However, this is again a contradiction by Lemma~\ref{lem:cycles}.
\end{proof}
Note that, by (1), the number of vertices in a $2$-connected $(0,1)$-ternary graph is $2$ modulo $3$.
Also, it was implicitly proved in (2) that no ear can connect an internal vertex of an ear to that of another ear.

Given a cycle $C$ and two distinct vertices $u$ and $v$ of $C$, let $C_{uv}^+$ and $C_{uv}^-$ be the two internally disjoint paths connecting $u$ and $v$ whose union is $C$.
We say a cycle $C$ is {\em nice} if there are two vertices $u$ and $v$ of $C$ such that all other vertices of $C$ have degree exactly $2$ in $G$, and the two paths $C_{uv}^+$ and $C_{uv}^-$ has length $k'$ modulo $m$.
\begin{lemma}\label{min-cycles}
Let $G$ be a $2$-connected graph that is not a cycle itself, where all cycles have length $k$ modulo $m$.
Then $G$ contains at least two nice cycles.
\end{lemma}
\begin{proof}
We proceed by induction on the number of ears in an ear decomposition of $G$.
Consider an ear decomposition $G = P_0 \cup \cdots \cup P_t$ of $G$ where $P_0$ is the initial cycle and $P_1,\ldots,P_t$ are the ears in the order in which they are attached.
Since $G$ is not a cycle, we may assume $t \geq 1$.
If $t=1$, then $G$ is obtained by attaching the ear $P_1$ to the initial cycle $P_0$, and the statement obviously holds.
Thus we may assume $t>1$.

By the induction hypothesis, $G' = P_0\cup P_1 \cup \cdots\cup P_{t-1}$ contains two nice cycles, say $C$ and $D$.
Since $G$ is not a cycle itself and it is $2$-connected, each of $C$ and $D$ should contain exactly two vertices of degree bigger than $2$ in $G$.
Say those vertices are $u$ and $v$ for $C$ and $w$ and $z$ for $D$.
Note that, since $u$ and $v$ has degree bigger than $2$, the $2$-connectivity assumption on $G$ guarantees the existence of a path from $u$ to $v$ that is internally disjoint from $C_{uv}^+$ and $C_{uv}^-$.
Since every internal vertex of $C_{uv}^+$ has degree $2$ in $G$, removing all internal vertices of $C_{uv}^+$ from $G$ results in a $2$-connected graph such that all cycles have length $k$ modulo $m$.
The same observation holds for each of $C_{uv}^-$, $D_{wz}^+$ and $D_{wz}^-$, implying that there exists an ear decomposition such that $C_{uv}^+$, $C_{uv}^-$,
$D_{wz}^+$ and $D_{wz}^-$ are ears.

Let the endpoints of $P_t$ are $x$ and $y$.
If both $x$ and $y$ are not internal vertices of one of $C_{uv}^+$, $C_{uv}^-$, $D_{wz}^+$ and $D_{wz}^-$, then the cycles $C$ and $D$ are nice in $G$ as well.
Thus we may assume that at least one of $x$ and $y$ is an internal vertex of one of the paths $C_{uv}^+$, $C_{uv}^-$,
$D_{wz}^+$ and $D_{ez}^-$.
Without loss of generality, assume $x$ lies on $C_{uv}^+$.
By Theorem~\ref{ear-decomp}, the two endpoints of the ear $P_t$ should lie on the same ear in the new ear decomposition.
This implies that $y$ also lies on $C_{uv}^+$ as well. 
If $C_{uv}^+$ is not a part of $D$, then $D$ and $C' = \xymatrix{x\ar[r]^{C_{uv}^+} & y\ar[r]^{P_t} & x}$ are nice in $G$. Otherwise, $C_{uv}^+$ is a part of $D$, meaning that it is equal to one of $D_{wz}^+$ and $D_{wz}^-$. Without loss of generality, assume $C_{uv}^+=D_{wz}^-$.
Then the cycles \(C_1'=\xymatrix{x\ar[r]^{C_{uv}^{+}} & y\ar[r]^{P_t} & x}\) and \(C_2'=C_{uv}^- + D_{wz}^+\) are nice in $G$.
In any case, we found two nice cycles in $G$.
\end{proof}

\section{Domination numbers of $(0,1)$-ternary graphs}\label{sec:i(G)}
In this section, we show that every connected $(0,1)$-ternary graph contains a vertex whose deletion does not decrease the independent domination number.
\begin{theorem}\label{one for all}
    For every connected $(0,1)$-ternary graph $G$, there exists a vertex $v$ such that $i(G-v) \geq i(G)$.
\end{theorem}
To see this, we investigate the behavior of $i(G)$-dominating set when $G$ is a $2$-connected $(0,1)$-ternary graph, based on the structure of such a graph $G$ described in Section~\ref{sec:ear-decomp}.
In the proof of Theorem~\ref{one for all}, we frequently use an obvious observation that $i(G)\geq i(G-v)+1$: if we take a minimum independent dominating set $I$ of $G-v$, then either $I$ also dominates $v$ in $G$ or $I\cup\{v\}$ is an independent dominating set of $G$. This holds for any graph.

Given a $2$-connected $(0,1)$-ternary graph $G$, let $\{u,v\} \subset V(G)$ be an {\it admissible pair} if one can obtain a $2$-connected $(0,1)$-ternary graph $H$ from $G$ by attaching an ear whose endpoints are precisely $u$ and $v$. Note that, by Theorem~\ref{ear-decomp}, two vertices forming an admissible pair must have distance $1$ modulo $3$.
Let {\em $i(G)$-dominating set} be a minimum independent dominating set of $G$.

\begin{theorem}\label{all for one}
    Let $G$ be a $2$-connected $(0,1)$-ternary graph on $3m+2$ vertices. Then 
    \begin{enumerate}
        \item $i(G) = m+1$,
        \item no $i(G)$-dominating set contains an admissible pair, and
        \item $i(G-x) \geq i(G)$ for all $x\in V(G)$.
    \end{enumerate}
    Moreover, for every $x\in V(G)$, there is an $i(G)$-dominating set containing $x$.
\end{theorem}

Before we prove Theorem~\ref{all for one}, we first show how we obtain Theorem~\ref{one for all} from Theorem~\ref{all for one}.

In a graph $G$, a {\em $2$-connected component} is a maximal $2$-connected subgraph.
Here, we assume that $K_2$ is $2$-connected.
It is known that any connected graph decomposes into a tree of $2$-connected components, which is often called as the {\em block-cut tree} of $G$.
Unless $G$ is already $2$-connected, the $2$-connected component $H$ of $G$ that corresponds to a leaf of the block-cut tree contains exactly one vertex, say the \textit{cut vertex}, that is adjacent to a vertex not in $H$.
\begin{proof}[Proof of Theorem~\ref{one for all}]
    Let $H$ be a $2$-connected component of $G$ that corresponds to a leaf in the block-cut tree of $G$. Let $v$ be the unique cut vertex in $H$, and let $J=G-H$.
    By Theorem~\ref{all for one} we can take an $i(H)$-dominating set $A$ that contains a neighbor of $v$. This particularly implies $v \notin A$. Then for any $i(J)$-dominating set $B$, the set $A\cup B$ is an independent dominating set of $G$, and hence we have $i(G) \leq i(J) + i(H)$.
    Since $v$ is a cut vertex, it must be $i(G-v) = i(J) + i(H-v)$.
    Then, by Theorem~\ref{all for one}, we have $i(G-v) \geq i(J) + i(H) \geq i(G)$.
    Therefore, the vertex $v$ is a desired vertex.
\end{proof}

In the rest of this section, we give a proof of Theorem~\ref{all for one}.

We need two lemmas. The first one is an observation of what happens in paths.
\begin{lemma}\label{path-dominating}
    Let $P$ be a path on $n$ vertices.
    Then $i(P) = \left\lceil\frac{n}{3}\right\rceil$.
    Also, if there is an \(i(P)\)-dominating set $I$ containing two vertices of distance \(1\) modulo \(3\), then $n\equiv 1~\text{(mod }3)$ and $I$ does not contain any of the endpoints of~\(P\).
\end{lemma}
\begin{proof}
    Let $P = v_1 v_2 \ldots v_n$.
    Clearly, $\{v_m: 1 \leq m \leq n, m\equiv 1~\text{(mod }3)\}$ is an independent set of size $\left\lceil\frac{n}{3}\right\rceil$ that dominates $P$.
    On the other hand, since every vertex of $P$ can dominate at most three vertices of $P$, we need at least $\left\lceil\frac{n}{3}\right\rceil$ vertices to dominate $P$.
    This shows that $i(P) = \left\lceil\frac{n}{3}\right\rceil$.

    Now suppose there is an $i(P)$-dominating set $I$ containing two vertices of distance $1$ modulo $3$.
    Let $i_1 < \cdots < i_{\lceil n/3 \rceil}$ be the indices of the vertices in $I$.
    Since $I$ is an independent set that dominates $P$, it must be $i_1\in\{1,2\}$, $i_{\lceil n/3 \rceil} \in \{n-1,n\}$, and $i_k - i_{k-1} \in \{2,3\}$ for all $2\leq k \leq \lceil n/3 \rceil$.
    In other words, if we set $i_0=-1$ and $i_{\lceil n/3 \rceil+1}=n+2$, then
    $$i_k - i_{k-1} \in \{2,3\}\;\;\forall 1\leq k \leq \left\lceil\frac{n}{3}\right\rceil+1.$$
    Since $I$ contains two vertices of distance $1$ modulo $3$, there are $2\leq k', k'' \leq \left\lceil \frac{n}{3} \right\rceil$ where $i_{k'}-i_{k'-1} = i_{k''}-i_{k''-1} = 2$. Then we have
    \begin{align*}
        (n+2)-(-1) &= \sum_{k=1}^{\lceil n/3 \rceil+1} (i_k-i_{k-1}) \leq 2\cdot2 + 3\cdot(\left\lceil\frac{n}{3}\right\rceil-1),
    \end{align*}
    that is, $n \leq 3\cdot\left\lceil\frac{n}{3}\right\rceil -2$.
    This particularly implies $n \equiv 1~\text{(mod }3)$.
    
    Observing that $n = 3\cdot\left\lceil\frac{n}{3}\right\rceil -2$ when $n \equiv 1~\text{(mod }3)$, it follows that $k'$ and $k''$ are the only values of $k$ where $i_k - i_{k-1} = 2$.
    That must happen between the two vertices of $I$ of distance $1$ modulo $3$, so it follows that $k', k'' \notin \{1, \left\lceil\frac{n}{3}\right\rceil\}$ and $i_1 = 2, i_{\left\lceil n/3 \right\rceil} = n-1$.
    Therefore, $I$ does not contain any endpoint of $P$.
\end{proof}

Given a path $P = v_1 \ldots v_n$, let $L_P = \{v_i\in P: i\equiv 1~\text{(mod }3)\}$, $M_P = \{v_i\in P: i\equiv 2~\text{(mod }3)\}$, and $R_P = \{v_i\in P: i\equiv 0~\text{(mod }3)\}$.
We denote by $uPv$ a path from a vertex $u$ to another vertex $v$ whose internal vertices form the path $P$.

In the second lemma, we introduce two modifications of an $i(G)$-dominating set under certain situations. Such modifications will be applied in the proof of Theorem~\ref{all for one}.
As usual, we denote the neighborhood and the closed neighborhood of a vertex $u$ in a graph $G$ by $N(u)=\{v\in V(G):uv\in E(G)\}$ and $N[u]=\{u\}\cup N(v)$, respectively.
\begin{lemma}\label{rearrangement}
    Let $G$ be a graph on $V$ and suppose $G$ contains a nice cycle $C=\xymatrix{u\ar[r]^{uPv} & v\ar[r]^{vQu} & u}$ where $|V(P)|=3p$, $|V(Q)|=3q$ and all vertices other than $u,v$ have degree exactly $2$ in $G$.
    Let $u'$ and $v'$ be the neighbors of $u$ and $v$ in $uPv$, respectively.
    If there is an $i(G)$-dominating set $I$ containing $u'$, then the following hold.
    \begin{enumerate}[(i)]
        \item If $v \in I$, then there is an $i(G)$-dominating set $I'$ where $I'\setminus V(Q) = I \setminus V(Q)$ and $I'\setminus V(P)$ dominates $G-P$.
        For convenience, when $V(P)=\varnothing$, we let $G-V(P)$ is the graph obtained from $G$ by deleting an edge connecting $u$ and $v$.
        \item If $N(v)\cap I = \{v'\}$, then there is an $i(G)$-dominating set $I'$ containing $u'$ and $v$ such that $I'\setminus(V(P)\cup\{v\}) = I\setminus(V(P)\cup\{v\})$.
        Note that if $V(P)=\varnothing$, $u' = v$ and $v' = u$, so this case does not happen.
    \end{enumerate}
\end{lemma}
\begin{proof}
    (i) This is obvious if $V(Q)=\varnothing$. Otherwise, let $V(Q)\cap N(v)=\{v''\}$. Since $v$ dominates $v''$ and $u\notin I$, there are $3q-1$ vertices of $Q$ that are not dominated by $I\setminus V(Q)$. 
    Thus it must be $|I\cap V(Q)| \geq q$.
    Noting that the vertices in $Q$ can be dominated by a subset of $\{u,v\}\cup V(Q)$ only, we see that $|I\cap V(Q)|$ equals to the independent domination number of a path $P_{3q-1}$, which is precisely $q$.
    This also can be achieved by replacing $I\cap V(Q)$ with $R_Q$.
    Then $I' = (I\setminus V(Q)) \cup R_Q$ is an $i(G)$-dominating set, and $I'\setminus V(P)$ dominates $G-P$.

    (ii) Observe that if $u',v'\in I$, then $|I\cap V(P)| \geq p+1$. One can see that the equality holds since the vertices in $P$ can be dominated by a subset of $\{u,v\}\cup V(P)$ only.
    Let $I'$ be the set obtained from $I$ by replacing $I\cap V(P)$ with $L_P\cup\{v\}$.
    Clearly, $I'$ dominates $G$ and $|I'|=i(G)$.
    Since $N(v)\cap I = \{v'\}$, $I'$ is independent.
    Therefore, $I'$ is the desired $i(G)$-dominating set.
\end{proof}

Now we are ready to prove Theorem~\ref{all for one}.
\begin{proof}[Proof of Theorem~\ref{all for one}]
    We proceed by induction on $m$.
    When $m=0$, $G$ consists of two vertices $u$ and $v$ that are adjacent to each other, it is clear that $i(G)=1$, no $i(G)$-dominating set contains an admissible pair, and $i(G-u)=i(G-v)=i(G)=1$.
    Suppose $m\geq 1$ and assume the statements hold for all $2$-connected $(0,1)$-ternary graphs on at most $3(m-1)+2$ vertices.

    By Theorem~\ref{ear-decomp} and Lemma~\ref{min-cycles}, there is a nice cycle $\xymatrix{u\ar[r]^{uPv} & v\ar[r]^{vQu} & u}$ where $u$ and $v$ have degree bigger than $2$ and $|P|=3p, |Q|=3q$ for some $p$ and $q$.
    Let $u'$ and $v'$ be the neighbors of $u$ and $v$ in $uPv$, respectively.

    \begin{enumerate}
        \item Since $G-P$ is a $2$-connected $(0,1)$-ternary graph on $3(m-p)+2$ vertices, the induction hypothesis says $i(G-P) = m-p+1$.
    Therefore, it suffices to show that $i(G) = i(G-P)+p$.

    First, we have $i(G) \leq i(G-P) + p$ by taking an $i(G-P)$-dominating set and adding $M_P$. On the other hand, consider an $i(G)$-dominating set $I$. If $I$ contains both $u'$ and $v'$, then it follows that $|I\cap V(P)| \geq p+1$. Since $I\setminus V(P)$ dominates $G-P-\{u,v\}$, it must be $|I\setminus V(P)| \geq i(G-P)-1$, and we have $$|I| = |I\setminus V(P)|+|I\cap V(P)| \geq i(G-P)+p.$$
    Otherwise, without loss of generality, we may assume $u' \notin I$. Then $I \setminus V(P)$ should dominate $G-P-v$, so $|I\setminus V(P)| \geq i(G-P-v) \geq i(G-P)$. Clearly, the $3p-2$ vertices of $P \setminus \{u',v'\}$ is not dominated by $I\setminus V(P)$, so $|I\cap V(P)| \geq p$. Therefore, it must be $|I| \geq i(G-P) + p$.
    In any case, we obtain $|I| = i(G-P)+p$.\hfill$\blacksquare$

    \item Let $I$ be an $i(G)$-dominating set.
    Since there is no admissible pair consists of a vertex in $P$ and a vertex not in $V(P)\cup\{u,v\}$, it is sufficient to check that both $I\setminus V(P)$ and $I\cap (V(P)\cup\{u,v\})$ does not contain an admissible pair.

    Suppose $u',v' \notin I$. Then $I\setminus V(P)$ should dominate $G-P$ and at least $3p-2$ many vertices of $P$ are not dominated by $I\setminus V(P)$. This implies $|I\setminus V(P)| = i(G-P) = m-p+1$ and $|I\cap V(P)| = p$.
    Clearly, $I\setminus V(P)$ contains no admissible pair by the induction hypothesis on $G-P$. This particularly implies that at most one of $u$ and $v$ can be included in $I$ since $\{u,v\}$ is an admissible pair.
    If both are not in $I$, then $I\cap (V(P)\cup\{u,v\})$ is equal to $M_P$, and it definitely does not contain an admissible pair.
    Otherwise, without loss of generality, assume $u\in I$. Let $P'$ be the path induced by $V(P) \cup \{u\}$. Then $I\cap (V(P)\cup\{u,v\})$ is an $i(P')$-dominating set containing $u$, so it does not contain an admissible pair.
    In any case, $I$ contains no admissible pair.

    Now, we may assume that at least one of $u'$ and $v'$ belongs to $I$. Without loss of generality, assume $v' \in I$.
    If $v'\in I$ and $u,u'\notin I$, then $I\setminus V(P)$ should dominate $G-P-v$, so $|I\setminus V(P)| \geq i(G-P-v) \geq i(G-P)$.
    However, since $v'$ dominates only $2$ vertices of $P$, $P$ still have $3p-2$ vertices that are not dominated by $(I \setminus V(P)) \cup \{v'\}$, and we need at least $p$ many vertices to dominate those.
    This implies $|I\cap V(P)| \geq p+1$, so we have $|I| \geq i(G-P)+p+1 = (m-p+1)+p+1 > m+1$, which is a contradiction.
    Therefore, it must be either $u \in I$ or $u' \in I$.

    Suppose $u\in I$. 
    By Lemma~\ref{rearrangement}~(i), we can find an $i(G)$-dominating set $I'$ that agrees with $I$ on $G-Q$ and $I'\setminus V(P)$ dominates $G-P$. By the induction hypothesis, $I \setminus V(Q) = I'\setminus V(Q)$ contains no admissible pair. Thus it suffices to show that $I \cap (V(Q)\cup\{u\})$ does not have any admissible pair. This follows from the fact that $I \cap (V(Q)\cup\{u\})$ is an $i(Q')$-dominating set where $Q'$ is the path induced by $V(Q)\cup\{u\}$ with length $0$ modulo $3$ .

    Finally, if $u'\in I$, then we first observe that $|I\cap V(P)| \geq p+1$.
    On the other hand, since $I\setminus V(P)$ dominates $G-P-\{u,v\}$, we have $|I\setminus V(P)|\geq i(G-P)-1$. As $|I| = i(G-P)+p$, it must be $|I\cap V(P)| = p+1$ and $|I\setminus V(P)|= i(G-P)-1$.
    Then by Lemma~\ref{path-dominating}, $I\cap V(P)$ does not contain an admissible pair. 
    Now if $N(u)\setminus \{u'\}$ meets $I$, then $I\setminus V(P)$ actually dominates $G-P-u$, implying that $|I\setminus V(P)| \geq i(G-P)$. This cannot happen, so we may assume that $N(u)\cap I = \{u'\}$. 
    By Lemma~\ref{rearrangement}~(ii), there is an $i(G)$-dominating set $I'$ containing $u$ that agrees with $I$ on $G-P-u$. This leads us to the previous case, implying that $I\setminus V(P)$ does not contain an admissible pair.
    This completes the proof of part (2).\hfill$\blacksquare$

    \item We may assume that $x \notin V(P)\cup V(Q)$. Otherwise, it must be either $x \in V(P)$ or $x \in V(Q)$. Assuming $x \in V(P)$ without loss of generality, we could start with another nice cycle that avoids $P$, so that $x$ is not a vertex of degree $2$ of the cycle.
    It is implied by Lemma~\ref{min-cycles} that such a cycle always exists.
    For contrary, suppose $i(G-x) \leq i(G)-1 = i(G-P)+p-1$. Let $I$ be an $i(G-x)$-dominating set.

    If $x = u$, then $I\cap V(P)$ should dominate $P-v'$, so $|I\cap V(P)| \geq p$. Since $I\setminus V(P)$ dominates $G-P-\{u,v\}$, we have $|I\setminus V(P)| \geq i(G-P-\{u,v\}) \geq i(G-P)-1$. By the assumption, it must be $|I\setminus V(P)| = i(G-P)-1$ and $|I\cap V(P)| = p$. The latter particularly implies that $v' \notin I$ by Lemma~\ref{path-dominating}, but then $I\setminus V(P)$ should dominate $G-P-u$. Therefore, we have $|I\setminus V(P)| \geq i(G-P-u) \geq i(G-P)$ by the induction hypothesis, but this is a contradiction. Thus we may assume $x \neq u$. By a symmetric argument, we may also assume $x \neq v$.

    Now, if $u',v' \notin I$, then we have $|I\setminus V(P)| \geq i(G-P)$ since $I\setminus V(P)$ should dominate $G-P-x$. 
    On the other hand, $V(P)\setminus\{u',v'\}$ still requires $p$ many vertices to be dominated, and this implies $i(G-x) \geq i(G-P)+p = i(G)$.
    If $u'\in I$ and $v'\notin I$, then we may assume $v \in I$. Otherwise, $V(P)$ requires $p+1$ many vertices to be dominated, and $I\setminus V(P)$ should dominate $G-P-x$, so by the induction hypothesis, we have $|I \setminus V(P)| \geq i(G-P-x) \geq i(G-P)$.
    Then $|I| \geq i(G-P) + p + 1 > i(G)$.
    Now, since $v\in I$, we can apply Lemma~\ref{rearrangement}~(i), to find an $i(G-x)$-dominating set $I'$ where $u'\in I'$, $v'\notin I'$, and $I'\setminus V(P)$ dominates $G-P-x$. By the induction hypothesis, we have $|I'\setminus V(P)| \geq i(G-P)$. Since $3p-1$ many vertices of $P$ are not dominated by $I'\setminus V(P)$, it must be $|I'\cap V(P)|\geq p$, so $|I|=|I'| \geq i(G-P)+p = i(G)$.
    
    Finally, suppose both $u'$ and $v'$ belong to $I$. Then we know $|I \cap V(P)| \geq p+1$ by Lemma~\ref{path-dominating}. 
    If $N(v) \cap I \neq \{v'\}$, then $I\setminus V(P)$ dominates $G-P-\{u,x\}$, so $|I\setminus V(P)|\geq i(G-P)-1$. It follows that $|I| \geq (p+1)+i(G-P)-1 = i(G)$.
    Otherwise, we have $N(v) \cap I = \{v'\}$, so we can apply Lemma~\ref{rearrangement}~(ii) to find an $i(G-x)$-dominating set $I'$ such that $u',v \in I'$. By Lemma~\ref{rearrangement}~(i), there is an $i(G-x)$-dominating set $I''$ where $I''\setminus V(P)$ dominates $G-P-x$. Similarly as above, this implies $|I|=|I''|\geq i(G)$.
    \hfill$\blacksquare$
    \end{enumerate}
    Finally, we show that for every $x\in V(G)$, there is an $i(G)$-dominating set containing $x$.
    If $x\notin P$ or $x\in M_P$, then we take the union of an $i(G-P)$-dominating set containing $x$ and $M_P$.
    If $x\in L_P$, then we take the union of $L_P$ and an $i(G-P)$-dominating set containing $v$.
    If $x\in R_P$, then we take the union of $R_P$ and $i(G-P)$-dominating set containing $u$.
    In any case, we constructed an independent dominating set of $G$ of size $i(G-P)+p=i(G)$.
    This completes the proof.
\end{proof}

Next, we give a lemma that compares the domination number of a $(0,1)$-ternary graph and that of the line graph of the graph.
\begin{lemma}\label{lem:upper}
If $G$ is $(0,1)$-ternary, then $\gamma(L(G)) \leq \gamma(G)$.
\end{lemma}
\begin{proof}
We prove by induction on $n=\gamma(G)$.
If $\gamma(G)=1$, then $G$ must be a star.
That is, $G$ has a vertex $v$ that is adjacent to all other vertices, and no pair of vertices other than $v$ are adjacent.
In this case, we obviously have $\gamma(L(G))=1=\gamma(G)$.
Assume $n>1$ and the statement holds for all $k<n$.
Take a minimum domination set $W$ of $G$.

If there is $w \in W$ that is isolated, then $\gamma(G) = \gamma(G-w)+1$. 
By the induction hypothesis, we have $\gamma(G)>\gamma(G-w) \geq \gamma(L(G-w)) = \gamma(L(G))$, where the last equality occurs since $w$ is isolated.
Thus we may assume every vertex in $W$ is incident with some edge.

Suppose that there exist nonempty sets $W'=\{w_1,\ldots,w_k\}\subset W$ and $F=\{e_1,\ldots,e_k\}$ such that $e_i$ is incident with $w_i$ for each $i \in [k]$ and $W\setminus W'$ dominates $G_F$, where $G_F$ is the graph obtained by removing all edges that meet an edge of $F$ and then removing all isolated vertices.
If $W' = W$, then there should be no vertices and edges in $G_F$.
In particular, it follows that $F$ dominates $L(G)$, so we obtain $\gamma(L(G)) \leq |F| = |W| = \gamma(G)$.
Otherwise, we have $\varnothing\neq W'\subsetneq W$.
Since $G_F$ is dominated by $W\setminus W'$, we have $\gamma(G_F)\leq |W|-|W'| = \gamma(G)-|W'|$.
Also, by the definition of $G_F$, it must be $\gamma(L(G)) \leq \gamma(L(G_F))+|F|$.
Since we know $\gamma(L(G_F)) \leq \gamma(G_F)$ by the induction hypothesis, it follows that $$\gamma(L(G)) \leq \gamma(L(G_F))+|F| \leq \gamma(G_F)+|W'| \leq \gamma(G).$$
Therefore, we may assume that for every choice of $W'=\{w_1,\ldots,w_k\}\subset W$ and $F=\{e_1,\ldots,e_k\}$ such that $e_i$ is incident with $w_i$, the graph $G_F$ contains a vertex $u$ that 
\begin{itemize}
    \item is not dominated by $W\setminus W'$,
    \item is incident to an edge $e=uv$ which is not dominated by $F$ in $L(G)$, i.e. $e=uv\in E(G_F)$, and
    \item is adjacent to some vertex in $W'$,
\end{itemize}
where the second and the third follow from the definition of $G_F$ and the assumption that $W$ dominates $G$, respectively.
Note that it must be $v\notin W\setminus W'$ because $u$ is not dominated by $W\setminus W'$. Indeed, this implies $v\notin W$ since the edge $e$ does not meet any edge in $F$.

Let $W=\{w_1,w_2,\ldots,w_n\}$.
Since $w_1$ is not an isolated vertex in $G$, there is an edge, say $e_1 = w_1 v_0$, that is incident to $w_1$.
Applying the above argument for $W'=\{w_1\}$ and $F=\{e_1\}$, there exists a vertex $u_1 \notin W$ that is not dominated by $W\setminus W'$ and $v_1 \notin W$ that is adjacent to $u_1$ such that $w_1u_1 \in E(G)$, $u_1 v_1\in E(G_F)$.
Since $G$ has no triangle, $v_1$ is not adjacent to $w_1$ and is not equal to $v_0$.
On the other hand, $v_1$ must have a neighbor in $W$.
Without loss of generality, assume $e_2=w_2v_1 \in E(G)$.
Then the path $w_1u_1v_1w_2$ has length $3$ and connects $w_1$ and $w_2$, where the third edge $v_1w_2$ belongs to $\{e_1,e_2\}$.
Since $G$ contains no cycles of length $0$ or $1$ modulo $3$, we can repeat such a process as in the following claim.
\begin{claim}\label{claim1}
For each integer $k \geq 1$, by relabeling the vertices in $W$ if necessary, there are $\{w_1,\ldots,w_k\}\subset W$, $\{e_1,\ldots,e_k\}\subset E(G)$, and vertices $v_0,\ldots,v_{k-1}$, $u_1,\ldots,u_{k-1}\in V(G)$ satisfying the following conditions:
\begin{enumerate}
\item $e_i=w_iv_{i-1}$ for each $i \in [k]$,
\item $u_i$ is dominated by $\{w_1,\ldots,w_i\}$ for each $i \in [k-1]$, 
\item $v_i$ is not dominated by $\{w_1,\ldots,w_i\}$ for each $i \in [k-1]$, 
\item $u_iv_i$ is not dominated by $\{e_1,\ldots,e_i\}$ in $L(G)$ for each $i \in [k-1]$,
\item for every $i\neq j$ in $[k]$, 
there is a path $P=p_0 p_1 \ldots p_{3m}$ from $w_i$ to $w_j$ such that 
$$\{p_{3t-3}p_{3t-2}:t\leq r\}\cup\{p_{3t+2}p_{3t+3}:t\geq r\} \subset \{e_1,\ldots,e_k\}$$ for some $0\leq r \leq m$, or there is a path $P=p_0 p_1 \ldots p_{3m-2}$ from $w_i$ to $w_j$ such that 
$$\{p_{3t-3}p_{3t-2}:t \leq m\} \subset \{e_1,\ldots,e_k\}.$$
\end{enumerate}
\end{claim}
\noindent
{\em Proof of Claim~\ref{claim1}.} 
We proceed by induction on $k$.
The above argument shows the base cases when $k=1$ and $k=2$.
Assume $k\geq2$ and suppose we have found $\{w_1,\ldots,w_k\}$, $\{e_1,\ldots,e_k\}$, $v_0,\ldots,v_{k-1}$, and $u_1,\ldots,u_{k-1}$ satisfying all the conditions.

Let $W' = \{w_1,\ldots,w_k\}$ and $F=\{e_1,\ldots,e_k\}$.
Similarly as before, there exists a vertex $u_k$ that is not dominated by $W\setminus W'$ but has a neighbor in $W'$, say $w_i$, and a vertex $v_k$ such that the edge $u_kv_k$ is not dominated by $F$ in $L(G)$.
We show that $v_k$ is not adjacent to any of $W'$.
Suppose $v_k$ is adjacent to $w_j\in W'$.
Since $G$ has no triangle, it must be $w_i \neq w_j$.
By the induction hypothesis, there is a path $P$ satisfying the condition (5). We divide into two cases.

\noindent {\bf Case 1.} Suppose $P$ has length $3m$.
Since the edge $u_k v_k$ is not dominated by $F$ in $L(G)$, if any of $u_k$ and $v_k$ is on $P$, then it should be equal either to $p_{3t-1}$ for some $t \leq r$ or to $p_{3t+1}$ for some $t \geq r$.
Note that $u_k, v_k \neq p_{3r}$ since $p_{3r}$ has distance $0$ modulo $3$ from both endpoints of $P$.
If $u_k, v_k \in V(P)$, then the distance between $u_k$ and $v_k$ on $P$ is either $0$ or $2$ modulo $3$.
However, by adding the edge $u_kv_k$ to the subpath of $P$ connecting $u_k$ and $v_k$, we obtain a cycle of length $1$ or $0$ modulo $3$, which is a contradiction.
This implies that at most one of $u_k$ and $v_k$ can lie on $P$.
Suppose $u_k \notin V(P)$ but $v_k \in V(P)$.
Since the distance between $w_i$ and $v_k$ on $P$ is either $2$ or $1$ modulo $3$, we obtain a cycle of length $1$ or $0$ modulo $3$ by adding the path $w_i u_k v_k$ to the subpath of $P$ connecting $w_i$ and $v_k$, so this is a contradiction.
Similarly, if $u_k\in V(P)$ and $v_k\notin V(P)$, we obtain a cycle of length $1$ or $0$ modulo $3$ by adding the path $w_j v_k u_k$ to the subpath of $P$ connecting $w_j$ and $u_k$.
This shows that $u_k, v_k \notin V(P)$.
Now, adding the path $w_i u_k v_k w_j$ to $P$ gives us a cycle of length $0$ modulo $3$, which is again a contradiction.

\noindent {\bf Case 2.} Suppose $P$ has length $3m-2$.
In this case, if any of $u_k$ and $v_k$ is on $P$, then it should be equal to $p_{3t-1}$ for some $t\leq m$.
If $u_k, v_k \in V(P)$, they have distance $0$ modulo $3$ on $P$, so adding $u_kv_k$ to the subpath of $P$ connecting $u_k$ and $v_k$ give us a cycle of length $1$ modulo $3$.
Observe that for every $t\leq m$, the vertex $p_{3t-1}$ has distance $2$ modulo $3$ from $p_0$ and $p_{3m-2}$ on $P$.
This implies that if either $u_k \notin V(P)$ and $v_k \in V(P)$ or $u_k \in V(P)$ and $v_k\notin V(P)$, then we obtain a cycle of length $1$ modulo $3$ by adding $w_i u_k v_k$ or $w_j v_k u_k$.
Finally, if $u_k,v_k \notin V(P)$, by adding the path $w_i u_k v_k w_j$ to $P$, we have a cycle of length $1$ modulo $3$.
In any case, there is a cycle of length $1$ modulo $3$, which is a contradiction.

Therefore, we conclude that $v_k$ does not have a neighbor in $W'$.
It follows that $v_k$ is adjacent to a vertex in $W\setminus W'$.
Without loss of generality, we may assume $w_{k+1}$ is adjacent to $v_k$. Let $e_{k+1} = w_{k+1}v_k$.
By the choice of $w_{k+1}, e_{k+1}, v_k$ and $u_k$, the conditions from (1) to (4) are already satisfied with $\{w_1,\ldots,w_{k+1}\}\subset W$, $\{e_1,\ldots,e_{k+1}\}\subset E(G)$, and $v_0,\ldots,v_{k}$, $u_1,\ldots,u_{k}\in V(G)$, so it is left to show (5).
Given $w_j \in W'$, we will combine the path $P$ from $w_j$ to $w_i$ that satisfies (5) with the path $w_i u_k v_k w_{k+1}$ to obtain a path from $w_j$ to $w_{k+1}$ satisfying the condition (5).
There are several cases as follows.

\noindent{\bf Case I.}
Suppose $P$ has length $3m$.
It was implicitly shown in Case 1 that both $u_k$ and $v_k$ cannot be vertices of $P$.
If $u_k, v_k \notin V(P)$, then we extend the path $P$ to obtain a path $P' = Pu_kv_kw_{k+1}$ of length $0$ modulo $3$ that connects $w_j$ and $w_{k+1}$.
Otherwise, either $u_k$ or $v_k$ should be a vertex of $P$.
Suppose $u_k\in V(P)$.
Then either $u_k = p_{3t-1}$ or $u_k = p_{3t+1}$ for some $t$.
For the latter case, $u_k$ has distance $2$ modulo $3$ from $w_i$, so the edge $u_k w_i$ completes a cycle of length $0$ modulo $3$.
Thus it must be $u_k=p_{3t-1}$ for some $t$.
Then we can construct the path $P'=p_0 p_1\ldots p_{3t-1}v_k w_{k+1}$ connecting $w_j$ and $w_{k+1}$.
See Figure~\ref{fig:Case I_1} for an illustration of $P'$ in this case.
\begin{figure}[htbp]
    \centering
    \includegraphics[scale=1]{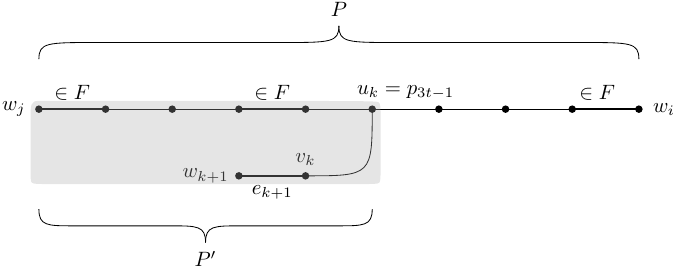}
    \caption{$P'=p_0 p_1\ldots p_{3t-1}v_k w_{k+1}$ when $P$ has length $3m$ and $u_k=p_{3t-1}$ for some $t$.}
    \label{fig:Case I_1}
\end{figure}
Now the only remaining case is when $v_k \in V(P)$ and $u_k\notin V(P)$.
If $v_k= p_{3t-1}$ or $v_k = p_{3t+1}$ for some $t$, then it has distance $1$ or $2$ modulo $3$ from $w_i$, so we can obtain a cycle of length $0$ or $1$ modulo $3$ by adding the path $w_i u_k v_k$ to the subpath of $P$ connecting $w_i$ and $v_k$.
Thus the only possiblity is that $v_k = p_{3r}$.
In this case, we construct $P'=p_0 p_1\ldots p_{3r}w_{k+1}$ of length $1$ modulo $3$ connecting $w_j$ and $w_{k+1}$.
See Figure~\ref{fig:Case I_2} for an illustration of $P'$ in this case.
\begin{figure}[htbp]
    \centering
    \includegraphics[scale=1]{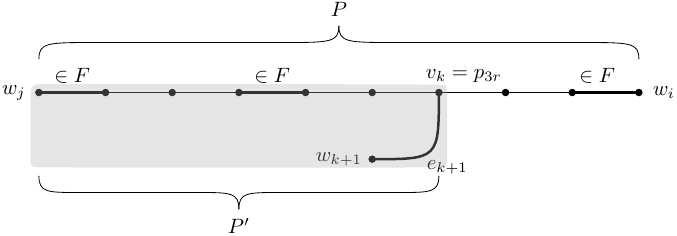}
    \caption{$P'=p_0 p_1\ldots p_{3r}w_{k+1}$ when $P$ has length $3m$ and $v_k=p_{3r}$.}
    \label{fig:Case I_2}
\end{figure}
In any case, since $v_kw_{k+1}=e_{k+1}$, the path $P'$ satisfies the condition (5).

\noindent{\bf Case II.}
Suppose $P$ has length $3m-2$.
If $u_k$ is on $P$, then it must be equal to $p_{3t-1}$ for some $t$. However, it has distance $2$ modulo $3$ from $w_i$, so the edge $u_kw_i$ completes a cycle of length $0$ modulo $3$.
Thus it must be $u_k \notin V(P)$.
In this case, it was shown in Case 2 that $v_k$ also cannot be a vertex of $P$. Therefore, we can extend the path $P$ to obtain a path $P' = Pu_kv_kw_{k+1}$ of length $1$ modulo $3$ that connects $w_j$ and $w_{k+1}$.
Since $v_kw_{k+1}=e_{k+1}$, the path $P'$ satisfies the condition (5).

For all cases, we could find a path from $w_j$ to $w_{k+1}$ satisfying the condition (5).
This completes the proof. \hfill$\blacksquare$

\medskip

Applying the above argument for $k=|W|$, we see that there exists a vertex $v$ that is not dominated by $W$.
This is a contradiction to the assumption that $W$ dominates $G$.
What we have shown is that there exist nonempty sets $W'=\{w_1,\ldots,w_k\}\subset W$ and $F=\{e_1,\ldots,e_k\}$ such that $e_i$ is incident with $w_i$ for each $i \in [k]$ and $W\setminus W'$ dominates $G_F$, and it is already shown that $\gamma(L(G))\leq\gamma(G)$.
This completes the proof.
\end{proof}

\section{Independence complex of $(0,1)$-ternary graphs and hypergraphs}\label{sec:main}
In this section, we prove the main theorems. 
We start with the graph version.
\begingroup
\def\thetheorem{\ref{thm:main}}
\begin{theorem}
For every $(0,1)$-ternary graph $G$, if $d(G)\neq *$ then $$d(G)+1=i(G)=\gamma(G)=\gamma(L(G)).$$
\end{theorem}
\addtocounter{theorem}{-1}
\endgroup
\begin{proof}
Suppose $G$ is a $(0,1)$-ternary graph where $I(G)$ is not contractible.

We first show $d(G) \geq i(G)-1$.
We proceed by induction on $n=|V(G)|$.
Suppose $G$ is a star where $v$ is the vertex that all other vertices are adjacent to $v$.
Clearly, $\{v\}$ dominates $G$, so we have $i(G)=1$.
On the other hand, $I(G)$ is the disjoint union of a simplex on $\{v\}$ and a simplex on $V(G)\setminus\{v\}$, so it is homotopy equivalent to a $0$-dimensional sphere.
Therefore, the statement is true when $G$ is a star.
When $n=2$, either $G$ is a star or $G$ consists of two isloated vertices.
In the latter case, $I(G)$ is contractible, so the statement holds for all graphs on at most $2$ vertices.
Let $n>2$ and suppose that the statement holds for all $(0,1)$-ternary graphs on less than $n$ vertices.
Let $G$ be a $(0,1)$-ternary graph on $n$ vertices that is not a star, and assume $I(G)$ is not contractible.

Take $v \in V(G)$. 
Note that $$I(G)=I(G-v) \cup I(G-N(v)), ~ I(G-v) \cap I(G-N(v))=I(G-N[v])$$ and $I(G-N(v))$ is contractible since it is a cone with apex $v$.
Thus we obtain the following exact sequence by the Mayer--Vietoris sequence:
\[ \cdots \to  \tilde{H}_i (I(G-N[v])) \to \tilde{H}_i(I(G-v)) \to \tilde{H}_i(I(G)) \to \tilde{H}_{i-1}(I(G-N[v])) \to \cdots.\]
By the induction hypothesis, we know that $\tilde{H}_{i}(I(G-v))=0$ for all $i \leq i(G-v)-2$ and $\tilde{H}_{i}(I(G-N[v]))=0$ for all $i \leq i(G-N[v])-2$.
Now by using the above exact sequence, to show $\tilde{H}_i(I(G))=0$ for all $i \leq i(G)-2$, it is enough to prove that $i(G-v) \geq i(G)$ and $i(G-N[v])+1 \geq i(G)$ for some vertex $v$ of $G$.
We can first observe that $i(G-N[v])+1 \geq i(G)$ holds for every vertex $v$ of $G$.
Let $W$ be an $i(G-N[v])$-dominating set. 
Then it is clear that $W \cup \{v\}$ is independent and dominating all vertices of $G$. Thus we obtain that $i(G) \leq i(G-N[v])+1$.
In addition, we can take a vertex $v$ of $G$ such that $i(G-v) \geq i(G)$ by Theorem~\ref{all for one}.
This completes the proof of the indequality $d(G)\geq i(G)-1$.

Now the inequality $d(G) \leq \gamma(L(G))-1$ follows from \cite[Theorem~1.2]{HW14}, so by Lemma~\ref{lem:upper}, we have $d(G)\leq \gamma(L(G))-1 \leq \gamma(G)-1 \leq i(G)-1$. 
It immediately follows that $d(G)+1=i(G)=\gamma(G)=\gamma(L(G))$.
\end{proof}

\subsection{$(0,1)$-ternary hypergraphs}\label{sec:hypergraphs}
Now we prove a hypergraph version of the main theorem.

For a hypergraph $H$ and an edge $e=\{v_1,\ldots,v_k\}$ of $H$, the hypergraph $H_e$ is obtained from $H$ by removing the edge $e$, introducing new vertices $w,u_1,\ldots,u_k$, and adding new edges $\{w,u_1\},\ldots,\{w,u_k\}$, $\{u_1,v_1\},\ldots,\{u_k,v_k\}$.
See Figure~\ref{fig:H_e} for an illustration.
\begin{figure}[htbp]
    \centering
    \includegraphics[scale=0.9]{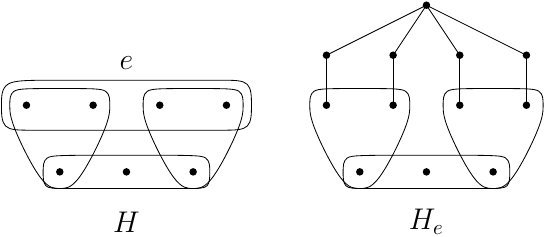}
    \caption{Constructing $H_e$ from a hypergraph $H$.}
    \label{fig:H_e}
\end{figure}
The following analogue of \cite[Theorem 11]{Cso09} was given in \cite{Kim24}.
Here, $\Sigma X$ is the suspension of the complex $X$.

\begin{theorem}\cite{Kim24}\label{thm:kim24}
    For a hypergraph $H$ and an edge $e$ of $H$, $I(H_e) \simeq \Sigma I(H)$. If a hypergraph $H$ has no Berge cycle of length $0$ modulo $3$, then neither does $H_e$.
\end{theorem}
Theorem~\ref{thm:kim24} implies that if a hypergraph $H$ has no Berge cycle of length $0$ modulo $3$ and $d(H)\neq *$, then $d(H_e)=d(H)+1$.
Also, the second statement of Theorem~\ref{thm:kim24} can be extended for the hypergraph property of being $(0,1)$-ternary.
\begin{lemma}\label{lem:hyperternary}
    If $H$ is a $(0,1)$-ternary hypergraph, then so is $H_e$.
\end{lemma}
\begin{proof}
    Suppose $H_e$ is not $(0,1)$-ternary, that is, $H_e$ contains a Berge cycle $C$ of length $0$ or $1$ modulo $3$.
    Let $e=\{v_1,\ldots,v_k\}$ and let $C=x_1 e_1 x_2 e_2 \ldots x_m e_m$ where $x_1,\ldots,x_m$ are distinct vertices and $e_1,\ldots,e_m$ are distinct edges of $H_e$.
    If all edges of $C$ already exist in $H$, then $C$ is a Berge cycle of $H$, which is a contradiction to the assumption that $H$ is $(0,1)$-ternary.
    So we may assume that $C$ contains an edge in $H_e \setminus H$.
    Since $$H_e \setminus H=\{\{w,u_1\},\ldots,\{w,u_k\}, \{u_1,v_1\},\ldots,\{u_k,v_k\}\},$$
    $C$ must include exactly $4$ consecutive edges in $H_e \setminus H$.
    Without loss of generality, assume $e_1=\{u_1,v_1\}$, $e_2=\{w,u_1\}$, $e_3=\{w,u_2\}$, $e_4=\{u_2,v_2\}$ are the edges of $H_e\setminus H$ that are used in $C$.
    Then we have $x_1=v_1$, $x_2=u_1$, $x_3=w$, $x_4=u_2$, $x_5=v_2$.
    Now the Berge cycle $C'=x_1 e x_5 e_5 x_6 e_6 \ldots x_m e_m$ in $H$ has length $m-3 \equiv m \equiv 0\text{ or }1~(\text{mod }3)$, which is a contradiction.
    Therefore, $H_e$ must be a $(0,1)$-ternary hypergraph.
\end{proof}
Based on the statements above, we prove Theorem~\ref{thm:hypergraph}.
\begingroup
\def\thetheorem{\ref{thm:hypergraph}}
\begin{theorem}
For every $(0,1)$-ternary hypergraph $H$, either $d(H)=*$ or $d(H)=\gamma(H)-1$.
\end{theorem}
\addtocounter{theorem}{-1}
\endgroup
\begin{proof}
    By Theorem~\ref{thm:kim24}, it is sufficient to show that $\gamma(H_e)=\gamma(H)+1$.
    
    Let $e=\{v_1,\ldots,v_k\}$.
    To prove $\gamma(H_e) \geq \gamma(H)+1$ by contrary, assume that $S$ is a minimum dominating set of $H_e$ with $|S| \leq \gamma(H)$.
    Suppose $w\in S$. If $u_i \in S$ for some $i$, we replace $u_i$ with $v_i$.
    Since $u_i$ affects only to $w$ and $v_i$ while $w$ is already in $S$, the set $S$ after this replacement still dominates $H_e$.
    If $v_i$ was in the original $S$, then the replacement reduces the size of the set, which is a contradiction to the minimality.
    By repeating the argument, we may assume that $S \cap \{u_1,\ldots,u_k\}=\varnothing$.
    Then it is clear that $S\setminus\{w\}$ is a dominating set of $H$ of size $|S|-1 < \gamma(H)$, which is a contradiction. 
    Thus we may assume $w\notin S$.

    Note that $u_i$ can be dominated by only $u_i$ or $v_i$ in $H_e$.
    Note also that some $u_i$ must be included in $S$ to dominate $w$.
    By relabeling the vertices $u_1,\ldots,u_m$ and $v_1,\ldots,v_m$ if necessary, we may assume that $S$ contains $u_1,\ldots,u_m, v_{m+1},\ldots,v_k$ for some $1 \leq m \leq k$.
    For $u_i \in \{u_2,\ldots,u_m\}$, we may assume $v_i\notin S$.
    Otherwise, $u_i$ takes no role in $S$, because $u_i$ is dominated by $v_i$ and $w$ is dominated by $u_1$.
    Now we construct $S'$ from $S$ by replacing each $u_i \in \{u_2,\ldots,u_m\}$ with $v_i$ and remove $u_1$.
    Since $v_1$ is dominated by $\{v_2,\ldots,v_k\}\subset S'$ in $H$, the set $S'$ is a dominating set of $H$.
    However, we have $|S'|=|S|-1<\gamma(H)$, which is a contradiction.
    This shows $\gamma(H_e)\geq \gamma(H)+1$.

    Next we show $\gamma(H_e) \leq \gamma(H)+1$.
    Take a dominating set $W$ of $H$ with $|W|=\gamma(H)$.
    We divide into the following two cases.

    \noindent {\bf Case 1.} Suppose $|W \cap e|\neq k-1$. Let $W'= W\cup\{w\}$. Since $W \subset W'$ and $|W\cap e|\neq k-1$, all vertices of $H$ are dominated by $W'$ in $H_e$ as well.
    Since $w$ dominates $\{w,u_1,\ldots,u_k\}$, the set $W'$ dominates $H_e$.

    \noindent {\bf Case 2.} $|W \cap e|=k-1$. 
    Without loss of generality, we may assume that $v_1\notin W$. Let $W'=W\cup\{u_1\}$.
    Then all vertices of $H$ possibly except $v_1$ are dominated by $W$ in $H_e$ as well.
    Also, each $u_i\in\{u_2,\ldots,u_k\}$ is dominated by $v_i\in W$.
    Finally, $u_1$ dominates $\{w,u_1,v_1\}$.
    This shows that $W'$ dominates $H_e$.
    
    In any case, we have $\gamma(H_e)\leq |W'|=\gamma(H)+1$, as desired.

    To complete the proof, we repeatedly apply Theorem~\ref{thm:kim24} and the above argument for all edges of $H$ with size at least $3$.
    Let $G$ be the resulting hypergraph after we apply the process for all edges of $H$ with size at least $3$.
    Note that we can assume $H$ has no edge of size $1$, since it does not affect to $I(H)$ and $\gamma(H)$.
    Hence $G$ is a $(0,1)$-ternary graph by Lemma~\ref{lem:hyperternary}.
    If $H$ has $k$ edges of size at least $3$, then we have $\gamma(G) = \gamma(H)+k$ by the above argument and $I(G)\simeq \Sigma^k I(H)$ by Theorem~\ref{thm:kim24}.

    If $d(H)\neq*$, then it must be $d(G)\neq*$, and it follows from Theorem~\ref{thm:main} that $\gamma(H)+k=\gamma(G)=d(G)+1=d(H)+k+1$.
    Therefore, we have $d(H) = \gamma(H)-1$.
\end{proof}

\section{Remark}\label{sec:remark}
In general, for a ternary graph $G$, the value $d(G)+1$ may be different from any of $i(G)$, $\gamma(G)$, and $\gamma(L(G))$.
One simple example is $C_4$.
Obviously, $I(C_4)$ consists of two disjoint edges, so $d(C_4)=0$. On the other hand, $i(C_4)=\gamma(C_4)=\gamma(L(C_4))=2 \neq d(C_4)+1$.
Here, we give two more such examples.
Computing $d(G)$ in both examples are based on an observation that $I(G) \simeq I(G-u)$ if $N(u) \supset N(v)$ for some non-adjacent pair $u,v\in V(G)$.
This is not difficult to prove, and one can refer to \cite[Lemma 3.2]{Eng09}.
\begin{example}
    Consider the graph $A_k$ obtained from $P_{3k}$ by replacing each edge with an induced $C_4$. 
    See Figure~\ref{fig:A_1} for an illustration when $k=1$.
    \begin{figure}[htbp]
        \centering
        \includegraphics[scale=0.9]{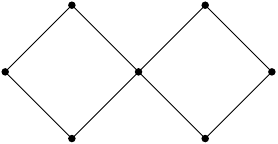}
        \caption{The graph $A_1$.}
        \label{fig:A_1}
    \end{figure}
    Then $I(A_k) \simeq I(P_{6k-1}) \simeq S^{2k-1}$, so $d(A_k) = 2k-1$.
    On the other hand, since each $C_4$ requires two vertices to be dominated, the most efficient way to dominate $A_k$ is to take the vertices that are shared by two $C_4$'s of $A_k$.
    There are $3k-2$ many such vertices in $A_k$, and we need $2$ more vertices in the first and the last $C_4$, so $\gamma(A_k)=i(A_k)=3k$.
    This shows that $d(G)$ and $\gamma(G)$, $i(G)$ may have an arbitrarily large gap.
\end{example}
\begin{example}
    Let $B_0$ be the graph obtained from $P_8$ by replacing each of the first, $4$th, and $7$th edge with an induced $C_4$.
    See Figure~\ref{fig:B_0} for an illustration of the graph.
    \begin{figure}[htbp]
        \centering
        \includegraphics[scale=0.9]{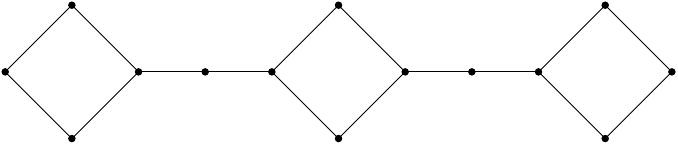}
        \caption{The graph $B_0$.}
        \label{fig:B_0}
    \end{figure}
    Then $I(B_0) \simeq I(P_{11})\simeq S^3$, so $d(B_0) = 3$.
    On the other hand, it is easy to see that $\gamma(B_0)=5$ and $\gamma(L(B_0)) = 6$.
    This not only shows that $d(G)$, $\gamma(G)$, and $\gamma(L(G))$ may be all distinct, but also implies that Lemma~\ref{lem:upper} may not hold if we allow cycles of length $1$ modulo $3$.

    We can also generalize this example to show that $d(G)$ and $\gamma(L(G))$ can have an arbitrarily large gap.
    Let $B_k$ is the graph obtained from $P_{9k+8}$ by replacing each of $(3t+1)$-st edge with an induced $C_4$.
    Then $I(B_k) \simeq I(P_{12k+11}) \simeq S^{4k+3}$, showing that $d(G)=4k+3$.
    On the other hand, we need at least $2$ edges to meet all edges of a $C_4$.
    Since the $C_4$'s in $B_k$ are far enough apart so that any edge that meets an edge in one copy of $C_4$ does not meet any edge of another copy of $C_4$.
    Noting that there are $3k+3$ many copies of $C_4$ in $B_k$, we have $\gamma(L(B_k)) \geq 6k+6$.
    On the other hand, construct an edge $F$ subset by taking a matching of size $2$ from each copy of $C_4$ in $B_k$.
    Then $|F| = 6k+6$ and $F$ meets all edges of $B_k$.
    This shows $\gamma(L(B_k)) = 6k+6$.
\end{example}
A natural problem that arises from the above observations is to find tight bounds on $d(G)$ for a ternary graph $G$ in terms of the domination numbers of $G$.

\subsection*{Funding}
T.~Eom and J.~Kim were supported by Global--Learning \& Academic research institution for Master's, PhD students, and Postdocs (LAMP) Program (RS-2024-00442775) of the National Research Foundation of Korea(NRF), funded by the Ministry of Education.

J.~Kim was also supported by the Institute for Basic Science (IBS-R029-C1) and by a grant(RS-2025-00558178), funded by the Ministry of Science and ICT(MSIT) of the Korean Government.

M.~Kim was supported by a grant(RS-2025-00555526) from the National Research Foundation of Korea(NRF), funded by the Ministry of Science and ICT(MSIT) of the Korean Government.

\end{document}